\documentclass[12pt]{amsart}

\usepackage{amsmath,amssymb,mathrsfs,amsthm}
\usepackage[utf8]{inputenc}
\usepackage{centernot}

\usepackage[utf8]{inputenc}

\usepackage{verbatim}
\usepackage{tikz}

\newcommand\NN{\mathbb{N}}

\theoremstyle{plain}
\newtheorem{theorem}{Theorem}[section]
\newtheorem{proposition}{Proposition}[section]

\newtheorem{lemma}{Lemma}[section]
\newtheorem{corollary}{Corollary}[section]

\theoremstyle{definition}
\newtheorem{definition}{Definition}[section]
\newtheorem{remark}{Remark}[section]

\newtheorem{question}{Question}[section]

\newcommand{\N}{\mathbb{N}} 

\newcommand{\e}{\varepsilon}

\def\e{\varepsilon}

\title{Growth orders and ergodicity for absolutely Ces\`{a}ro bounded operators}

\author[Luciano Abadias]{Luciano Abadias}
\address[L. Abadias]{Departamento de Matem\'aticas, Instituto Universitario de Matem\'aticas y Aplicaciones, Universidad de Zaragoza, 50009 Zaragoza, Spain.}
\email{labadias@unizar.es}

\author[Antonio Bonilla]{Antonio Bonilla}
\address[A. Bonilla]{Departamento de An\'alisis Matem\'atico, Universidad de la Laguna, 38271, La Laguna (Tenerife), Spain.}
\email{abonilla@ull.es}

\thanks{The first author has been partially supported by Project MTM2016-77710-P, DGI-FEDER, of the MCYTS and Project E-64, D.G. Arag\'on, Universidad de Zaragoza, Spain. The second author is partially supported by Project  MTM2016-75963-P, DGI-FEDER, of the MCYTS}

\subjclass[2010]{47A35, 47A10, 47B99.}

\keywords{Absolutely Ces\`aro bounded, Kreiss bounded, backward shift operator, ergodicity.}

\begin{document}

\begin{abstract}

In this paper, we extend the concept of absolutely Ces\`aro boundedness to the fractional case. We construct a weighted shift operator belonging to this class of operators, and we prove that if $T$ is an absolutely Ces\`{a}ro bounded operator of order $\alpha$ with $0<\alpha\le 1,$  then $\| T^n\|=o(n^{\alpha})$, generalizing the result obtained for $\alpha =1$. Moreover, if $\alpha > 1$, then $\|T^{n}\|= O(n)$. We apply such results to get stability properties for the Ces\`aro means of bounded operators.

\end{abstract}

\maketitle

\section{Introduction}

Let $X$ be a complex Banach space and $T\in\mathcal{B}(X),$ the Banach algebra of all bounded linear operators defined on $X.$ We denote by $\mathcal{T}$ the discrete semigroup given by the natural powers of the operator $T,$ that is, ${\mathcal T}(n):= T^n$ for $n\in \Bbb N_0$. Recall (\cite{Abadias, ALMV, Ed-dari 1, Ed-dari 2, LSAS,SZ}) that the \emph{Ces\`aro sum} of order $\alpha \geq0$
of $T$ is the family of operators $(\Delta ^{-\alpha}{\mathcal T}(n))_{n\in \Bbb N_0}\subset B(X)$ given by
$$
\Delta ^{-\alpha}{\mathcal T}(n) x:=(k^{\alpha}*\mathcal T)(n)x=\sum _{j=0}^n  k^{\alpha}(n-j) T^jx, \quad  x\in X,\, n\in \Bbb N_0,
$$
and the \emph{Ces\`aro mean} of order $\alpha\geq 0$ of $T$ is the family of operators $(M_T^{\alpha}(n))_{n\in \Bbb N_0}$ given by
$$
M_T^{\alpha}(n)x:= \frac{1}{k^{\alpha+1}(n)}\Delta ^{-\alpha}{\mathcal T}(n)x,\quad  x\in X,\, n\in \Bbb N_0,
$$
where $(k^{\alpha}(n))_{n\in\N_0}$ is the \emph{Ces\`aro kernel} of order $\alpha.$ It is given by $$k^{\alpha}(n):=\frac{\Gamma(\alpha+n)}{\Gamma(\alpha)\Gamma(n+1)},\quad n\in \Bbb N_0,\, \alpha>0,$$ where $\Gamma$ is the Gamma  function, and $k^0(n)=\delta_{n,0}$ for $n\in\N_0,$ where $\delta_{n,j}$ is the Kronecker delta, i.e., $\delta_{n,j}=1$ if $j=n$ and $0$ otherwise. We attach some of its properties that we will use along the paper. It satisfies that $k^{\alpha}*k^{\beta}=k^{\alpha+\beta},$ for $\alpha,\beta>0.$ Asymptotically, the behaviour of $k^{\alpha}$ is
 \begin{equation}\label{asymp}
   k^{\alpha}(n):=\frac{n^{\alpha-1}}{\Gamma(\alpha)}(1+O(1/n)),
 \end{equation}
and $k^{\alpha}$ is increasing (as a function of n) for $\alpha >1,$ decreasing  for $0<\alpha <1,$ and $k^{1}(n)=1$ for $n\in \Bbb N_0$. Moreover, the Gautschi inequality implies
\begin{equation}\label{gamma}
\frac{(n+1)^{\alpha -1}}{\Gamma(\alpha)}\leq k^{\alpha}(n)\leq \frac{n^{\alpha -1}}{\Gamma(\alpha)}, \;\; n\in \Bbb N, \;\; 0<\alpha \leq 1.
\end{equation}
For more details see \cite{Lizama1, Lizama, Zygmund}.

\bigskip

Recall that a bounded linear operator $T$ on $X$ is called \emph{$(C, \alpha)$-Ces\`{a}ro bounded} if the Ces\`{a}ro mean of order $\alpha$ of $T,$ $(M_T^{\alpha}(n))_{n\in \N_0},$ is uniformly bounded. The particular cases $\alpha=0$ and $\alpha=1$ are well-known, the power boundedness and Ces\`{a}ro boundedness respectively. It is known that $(C,\alpha)$-Ces\`{a}ro boundedness implies $(C,\beta)$-Ces\`{a}ro boundedness for $\beta>\alpha\geq 0,$ however the converse is not true in general: the Assani matrix is Ces\`{a}ro bounded but it is not power bounded. In the case that $(M_T^{\alpha}(n))_{n\in\N_0}$ converges in the strong topology of $X,$ we say that the operator $T$ is \emph{$(C, \alpha)$-mean ergodic} (for $\alpha =1$ it said to be mean ergodic). Properties, characterization thorough functional calculus and ergodic results for $(C,\alpha)$-Ces\`{a}ro bounded operators can be found in \cite{ALMV, AS, De00, E, Ed-dari 1, Ed-dari 2, LSAS} and references therein.

\bigskip

Also, there are connections between some boundedness of the Ces\`aro means and the resolvent of $T.$ We recall that the condition
$$
\|M^2_{\lambda T}(n)\| \le C,\;\; \mbox{ for } |\lambda |=1  \mbox{ and } n\in\N_0,
$$ is equivalent to the \emph{Kreiss bounded} condition
$$
\|(\lambda-T)^{-1}\| \le \frac{ C}{(|\lambda|-1)}, \;\; \mbox { for all } |\lambda |>1,
$$
and $$
\|M_{\lambda T}(n)\| \le C, \;\; \mbox{ for } |\lambda |=1  \mbox{ and } n\in\N_0,
$$ is equivalent to  the \emph{Uniformly  Kreiss bounded} condition
$$
\left\|\sum_{k=0} ^{n} \lambda^{-k-1} T^k\right\| \le \frac{ C}{(|\lambda|-1)}, \;\; \mbox { for all } |\lambda |>1,\, n\in\N_0,
$$
where $C$ is a positive constant. There exist Kreiss bounded operators which are not Ces\`{a}ro bounded, and reciprocally, Ces\`{a}ro boundedness does not imply Kreiss boundedness, see \cite{SZ}. Also, there are Kreiss bounded operators that they are not uniformly Kreiss bounded, see \cite{MSZ}. Moreover, in finite dimensional Hilbert spaces, Kreiss bounded and power bounded are equivalent classes of operators, see \cite{MSZ}.

\bigskip

Growth orders for the Ces\`aro means of any order (especially for the natural powers) of $(C,\alpha)$-Ces\`aro bounded or Kreiss bounded operators have been studied in a large amount of papers (\cite{Abadias, Ed-dari 1, LSAS, N, SZ} and references therein). In the following remark, we include some interesting particular cases.

\begin{remark}\label{kreiss2} Given $T\in X,$ $\sigma(T)\cap \partial \mathbb{D}$ denotes the spectrum of $T$ on the unit circle and $m(\sigma(T)\cap \partial \mathbb{D})$ its Lebesgue measure. From \cite{SZ}, \cite{N} and \cite{BermBMP}, it follows that if one of next conditions is satisfied:
\begin{enumerate}
\item [i)] $T$ is Ces\`aro bounded and $\sigma(T)\cap \partial \Bbb D \subset \{1\}.$

\item [ii)] $T$ is Kreiss  bounded and $m(\sigma(T)\cap \partial \Bbb D)=0$.

\item [iii)] $T$ is uniformly Kreiss  bounded and $X$ is a Hilbert space.

\end{enumerate}
Then  $\|T^n\|=o(n)$.

Furthermore, a similar result to the part i) is true for $\alpha\geq 1,$ that is, if $T$ is $(C,\alpha)$-Ces\`aro bounded with $\alpha\geq1$ and $\sigma(T)\cap \partial \Bbb D \subset \{1\},$ then $\|T^n\|=o(n^{\alpha})$ (\cite[Theorem 4.3]{Abadias}). Also, if we only consider that $T$ is Ces\`aro bounded or Kreiss  bounded, then $\|T^n\|=O(n)$ (\cite{Shi, SW}).
\end{remark}

\bigskip

Recently in \cite{HL}, the authors introduced a new concept about the growth of the Ces\`aro means. An operator $T\in\mathcal{B}(X)$ is called \emph{absolutely Ces\`aro bounded}, if there exists $C>0$ such that $$\sup_{n\in \N_0}\frac{1}{n+1}\sum_{j=0}^n \lVert T^jx\rVert\leq C\lVert x\rVert,\quad x\in X.$$ The starting point of the above concept is related to the distributional chaos theory. In \cite{BermBMP}, some examples of absolutely Ces\`aro bounded are given, as well as its connection with power bounded and Kreiss bounded properties. Also they prove that if $T$ is an absolutely Ces\`aro bounded operator, then $\|T^n\|=o(n).$ In this paper we introduce an extension of the absolutely Ces\`aro boundedness, and our main aim is to construct simple examples, to study the asymptotic behaviour of the orbits, to compare with other mentioned concepts, and to show ergodic results.

\begin{definition}
Let $\alpha>0.$ We say that a bounded linear operator $T$ on $X$ is \emph{absolutely $(C, \alpha)$-Ces\`aro bounded} if there exists a constant $C > 0$ such that
$$
\sup_{n \in \N_0} \frac{1}{k^{\alpha+1}(n)}\sum _{j=0}^n  k^{\alpha}(n-j) \|T^jx\|\le C\|x\|\;,
$$
for all $x\in X$.
\end{definition}
For $\alpha =1$, the above definition is the absolutely Ces\`aro boundedness. In order to clarify to the reader, we show the following sketch:
$$\text{Power bounded }\Rightarrow\text{ Absolutely }(C, \alpha)\text{-Ces\`aro bounded }$$
$$\Rightarrow\ (C, \alpha)\text{-Ces\`aro bounded  }\Rightarrow\ \|T^n\|=O(n^{\alpha})$$
Particularly, it follows that all absolutely Ces\`{a}ro bounded operators are uniformly Kreiss bounded, and that every absolutely $(C,2)$-Ces\`{a}ro bounded operator is Kreiss bounded.

\bigskip

The outline of this paper is as follows: In Section 2, for $0<\alpha \leq 1,$ we construct a class of weighted backward shift operators $(T_{\beta})_{\beta >0}$ on $\ell^p(\Bbb N),$ for  $1\le p < \infty,$ which are mixing absolutely $(C,\alpha)$-Ces\`{a}ro bounded with $\|T^n\|= (n+1)^{\beta}$ for $0<\beta<\alpha/p$ (Theorem \ref{ejemplos} and Corollary \ref{mixing}), and no $(C,\alpha)$-Ces\`aro bounded if $\beta\geq 1/p$ (Remark \ref{remark}). Moreover, we prove that all absolutely $(C,\alpha)$-Ces\`{a}ro bounded operators in a Banach space satisfy that $\|T^n\|=o(n^{\alpha})$ with $0<\alpha\le 1$ (Corollary \ref{ACB}), generalizing the result obtained in \cite{BermBMP} for $\alpha =1,$ and $\|T^n\|=O(n)$ for $\alpha>1$ (Corollary \ref{On}). In Hilbert spaces, for $0<\alpha\leq 1,$ if $T$ is an absolutely $(C,\alpha)$-Ces\`aro bounded operator then $\|T^n\|=o(n^{min\{\alpha,1/2\}})$ (Corollary \ref{acbhilbert}). In Section 3, we study some ergodic applications. We prove that the absolutely $(C,\alpha)$-Ces\`aro boundedness in a reflexive Banach space implies $(C,\alpha)$-mean ergodicity (Corollary \ref{ErgRef}). Also, we study stability results for the Cesàro means of fractional order of a bounded linear operator $T,$ assuming spectral and growth conditions on $T$ (Theorem \ref{ergodic} and Corollary \ref{ergodic2}).






\section{Examples and growth orders of absolutely Ces\`{a}ro bounded operators}

We denote by $\{e_n\}_{n\in \NN}$ the standard canonical basis on $\ell^p(\NN)$ for $1\leq p<\infty ,$ that is, $e_n=(\delta_{n, j})_{j\in \NN}:=(0, \ldots, 0,\underbrace{1}_{\substack{n}},0,\ldots),$ where $\ell^p(\N)$ is the Lebesgue space of complex sequences $x=\displaystyle \sum_{j=1}^\infty \alpha_j e_j ,$ with
$\displaystyle\Vert x\Vert_p:=\left(\sum_{j=1}^{\infty}\lvert \alpha_j\rvert^p\right)^{1\over p}<\infty.$

In the literature, there are only simple examples of $(C,\alpha)$-Ces\`{a}ro bounded operators (no power bounded) for $\alpha\in\N.$ In the following, we construct a class of $(C,\alpha)$-Ces\`{a}ro bounded backward shift operators on $\ell^p(\N),$ with $0<\alpha<1,$ which are not power bounded.

\begin{theorem}\label{ejemplos} Let $0<\alpha \le 1$, $1\leq p<\infty$ and $0<\beta <\frac{\alpha}{p}.$ The unilateral weighted backward shift operator $T,$ defined by $Te_1:=0$ and $Te_j:=w_je_{j-1}$ for $j>1,$ with $w_j:=\displaystyle \left( \frac{j}{j-1}\right)^{\beta}$,
is absolutely $(C, \alpha)$-Ces\`{a}ro bounded on $\ell^p(\NN)$.
\end{theorem}

\begin{proof}
Denote $ \varepsilon := \alpha-\beta p>0$. Let $x\in \ell^p(\NN)$ whose Fourier representation is $x:=\displaystyle \sum_{j=1}^\infty \alpha_j e_j .$ We assume without loss of generality that $\lVert x\rVert_p=1.$ For $N\in\NN$ we have

\begin{eqnarray}
 \sum_{n=0}^N k^{\alpha}(N-n)\|T^nx\|_p^p &=& \sum_{n=0}^N k^{\alpha}(N-n)\sum_{j=n+1}^\infty |\alpha_j|^p\Big(\frac{j}{j-n}\Big)^{\alpha-\varepsilon}   \nonumber \\
 &=& \sum_{j=1}^\infty |\alpha_j|^p\,  \sum_{n=0}^{\min\{N,\;j-1\}}k^{\alpha}(N-n)\Big(\frac{j}{j-n}\Big)^{\alpha-\varepsilon} \nonumber \\
 &=& \sum_{j=1}^{N} |\alpha_j|^p\, j^{\alpha-\varepsilon} \sum_{n=0}^{j-1}k^{\alpha}(N-n)({j-n})^{\varepsilon-\alpha} \\
 &+& \sum_{j=N+1}^{2N} |\alpha_j|^p\, j^{\alpha-\varepsilon} \sum_{n=0}^{N}k^{\alpha}(N-n)({j-n})^{\varepsilon-\alpha} \\
  &+& \sum_{j=2N+1}^\infty |\alpha_j|^p \sum_{n=0}^N k^{\alpha}(N-n)\Big(\frac{j}{j-n}\Big)^{\alpha- \varepsilon}.
 \end{eqnarray}
In what follows, we estimate each of the above summands.

First, notice that for $j\geq 2N+1$ and $n\leq N,$ one gets
$$
\left( \frac{j}{j-n} \right)^{\alpha-\varepsilon}\leq 2^{\alpha-\varepsilon}<2^{\alpha}\;.
$$
Hence
$$
(5)<2^{\alpha}k^{\alpha+1}(N)\sum_{j=2N+1}^\infty  |\alpha _j|^p\leq 2^{\alpha}k^{\alpha+1}(N) \;.
$$

Secondly, taking $j\le N$ to estimate the summand (3), we obtain
\begin{eqnarray*}
\sum_{n=0}^{j-1}k^{\alpha}(N-n)({j-n})^{\varepsilon-\alpha} &= & \sum_{n=1}^{j}k^{\alpha}(N-j+n){n}^{\varepsilon-\alpha} \\
&\leq &  \frac{1}{\Gamma(\alpha)}\sum_{n=1}^{j}(N-j+n)^{\alpha-1}{n}^{\varepsilon-\alpha} \\
&\le & \frac{1}{\Gamma(\alpha)}\sum_{n=1}^{j} n^{\varepsilon -1} < \frac{1}{\Gamma(\alpha)}( 1+ \int _1^{j} x^{\varepsilon -1}dx) \\
&=& \frac{1}{\Gamma(\alpha)}\frac{j^\varepsilon }{\varepsilon},
\end{eqnarray*}
where we have used \eqref{gamma}.

Now, in order to estimate (4) we take $j> N,$ and we have
\begin{eqnarray*}
\sum_{n=0}^{N}k^{\alpha}(N-n)({j-n})^{\varepsilon-\alpha} &= & \sum_{n=0}^{N}k^{\alpha}(n)({j-N+n})^{\varepsilon- \alpha} \\
&\leq &\frac{1}{\Gamma(\alpha)}\sum_{n=1}^{N}n^{\alpha -1}({j-N+n})^{\varepsilon- \alpha} +1\\
&\le&\frac{1}{\Gamma(\alpha)}\sum_{n=1}^{j-1} n^{\varepsilon -1} +1 < \frac{1}{\Gamma(\alpha)}( 1+ \int _1^{j-1} x^{\varepsilon -1}dx) +1\\
&=& \frac{1}{\Gamma(\alpha)}\frac{(j-1)^\varepsilon}{\varepsilon}+1 <\frac{1}{\Gamma(\alpha)}\frac{j^\varepsilon }{\varepsilon} +1\;.
\end{eqnarray*}

Using the previous estimates we get
\begin{eqnarray*}
\sum_{n=0}^N k^{\alpha}(N-n)\| T^nx\|^p_p
&\leq & \sum_{j=1}^{N} |\alpha _j|^pj^{\alpha-\varepsilon} \frac{1}{\Gamma(\alpha)}\frac{j^\varepsilon }{\varepsilon}\\
&& + \sum_{j=N}^{2N} |\alpha _j|^pj^{\alpha-\varepsilon}(1 + \frac{ j^{\varepsilon }}{\Gamma(\alpha)\varepsilon})+ 2^{\alpha}k^{\alpha+1}(N) \\
&=&  \sum_{j=1}^{2N} |\alpha _j|^p(1 +\frac{1}{\Gamma(\alpha)\varepsilon})j^{\alpha} +2^{\alpha}k^{\alpha+1}(N) \\
&\leq& (1+\frac{1}{\Gamma(\alpha)\varepsilon}) (2N)^{\alpha }\sum_{j=1}^{2N} |\alpha _j|^p + 2^{\alpha}k^{\alpha+1}(N)\\
 &\leq&  C(\alpha, \varepsilon) k^{\alpha+1}(N),
\end{eqnarray*}
and by Jensen's inequality
$$
\left( \frac{1}{k^{\alpha+1}(N)} \sum_{n=0}^N k^{\alpha}(N-n)\|T^nx\|_p\right)^p \leq \frac{1}{k^{\alpha+1}(N)} \sum_{n=0}^N k^{\alpha}(N-n)\| T^nx\|_p^p\leq  C(\alpha, \varepsilon).
$$

\end{proof}

Recall that a bounded linear operator $T$ on $X$ is \emph{topologically mixing} if for any pair $U,V$ of non-empty
open subsets of $X$, there exists $n_0 \in \NN$ such that $T^n(U) \cap V \neq \emptyset$ for all $n \geq n_0$. Therefore, we deduce the following result.

\begin{corollary}\label{mixing}
Let $1\leq p<\infty$  and $0<\varepsilon<\alpha\le 1.$ There exists an absolutely $(C, \alpha)$-Ces\`{a}ro bounded operator $T$ on $\ell^p(\NN)$ such that it is mixing  and $\|T^n\| =(n+1)^{\frac{(\alpha-\varepsilon)}{p}}$.
\end{corollary}
\begin{proof}
By  Theorem \ref{ejemplos} we have that the unilateral weighted backward shift operator on $\ell ^p(\NN)$ is absolutely $(C, \alpha)$-Ces\`{a}ro bounded and
\begin{equation}\label{ejl}
\| T^n\|=(n+1)^{\frac{(\alpha-\varepsilon)}{p}} \;.
\end{equation}
Moreover, by \cite[Theorem 4.8]{GEP11} we have that $T$ is mixing since $\left( \prod_{k=1}^n w_k \right)^{-1} \to 0$ as $n\to\infty$.
\end{proof}

The unilateral weighted backward shift operator on $\ell ^p(\NN),$ with $\beta=\frac{1}{p},$ given in Theorem \ref{ejemplos}, is not Ces\`aro bounded, and so it is not $(C,\alpha)$-Ces\`aro bounded for $0<\alpha<1,$ see \cite{BermBMP}. It is natural to ask if this operator is $(C,\alpha)$-Ces\`aro bounded for some $\alpha>1.$ The answer is negative as the following proposition shows.

\begin{proposition} \label{propos 2.1} The unilateral weighted backward shift operator $T$ on $\ell ^p(\NN)$ with $1\leq p<\infty,$ given by $Te_1:=0$ and $Te_j:=w_je_{j-1}$ for $j>1,$ with $w_j:=\displaystyle \left( \frac{j}{j-1}\right)^{1/p},$ is not $(C, \alpha)$-Ces\`{a}ro bounded for any $\alpha$.
\end{proposition}
\begin{proof}
By the above comment it is enough to prove the result for $\alpha>1.$ Let $N$ an even natural number. We define $y_{N+1}:=\frac{1}{(N+1)^{1/p}}\displaystyle\sum_{l=0}^{N+1}e_{l}.$ Then

\begin{displaymath}\begin{array}{l}
\displaystyle\|\frac{1}{k^{\alpha+1}(N)}\sum_{j=0}^N k^{\alpha}(N-j)T^jy_{N+1}\|_p^p\\
\displaystyle=\|\frac{1}{k^{\alpha+1}(N)(N+1)^{1/p}}\sum_{l=1}^{N+1}\biggl(\sum_{j=l}^{N+1} k^{\alpha}(N+l-j)(j/l)^{1/p}\biggr)e_{l}\|_p^p\\
\displaystyle=\frac{1}{(k^{\alpha+1}(N))^p(N+1)}\sum_{l=1}^{N+1}\frac{1}{l}\biggl(\sum_{j=l}^{N+1} k^{\alpha}(N+l-j)j^{1/p}\biggr)^p\\
\displaystyle\geq  \frac{1}{(k^{\alpha+1}(N))^p(N+1)}\sum_{l=1}^{N/2+1}\frac{1}{l}\biggl(\sum_{j=N/2+1}^{N+1} k^{\alpha}(N+l-j)j^{1/p}\biggr)^p \\
\displaystyle\geq  \frac{1}{(k^{\alpha+1}(N))^p(N+1)}\sum_{l=1}^{N/2+1}\frac{1}{l}\biggl(\sum_{j=N/2+1}^{N+1} k^{\alpha}(N+1-j)j^{1/p}\biggr)^p \\
\displaystyle=  \frac{1}{(k^{\alpha+1}(N))^p(N+1)}\sum_{l=1}^{N/2+1}\frac{1}{l}\biggl(\sum_{j=0}^{N/2} k^{\alpha}(N/2-j)(j+N/2+1)^{1/p}\biggr)^p \\
\displaystyle\geq  \frac{1}{(k^{\alpha+1}(N))^p(N+1)}\sum_{l=1}^{N/2+1}\frac{1}{l}\biggl(\sum_{j=0}^{N/2} k^{\alpha}(N/2-j)(N/2)^{1/p}\biggr)^p \\
\displaystyle\geq  C \sum_{l=1}^{N/2+1}\frac{1}{l}\geq C\ln (N/2+1),
\end{array}\end{displaymath}
with $C$ a positive constant, where we have used that $k^{\alpha}$ is increasing as function of $n$ for $\alpha >1,$ identity \eqref{asymp} and $k^{\alpha}*k^1=k^{\alpha+1}$. So, we conclude that the Ces\`aro mean of order $\alpha$ of $T$ is not uniformly bounded.
\end{proof}

\begin{remark}\label{remark} Let $0<\alpha \le 1$, $1\leq p<\infty,$ $\beta>0$ and $T$ be the unilateral weighted backward shift operator on $\ell ^p(\NN)$  defined in Theorem \ref{ejemplos}. Then
the operator $T$ is  absolutely $(C,\alpha)$-Ces\`aro bounded if $\beta<\alpha/p$ (Theorem \ref{ejemplos}) and no $(C,\alpha)$-Ces\`aro bounded if $\beta\geq 1/p$ (Proposition \ref{propos 2.1}).
\end{remark}

From now to the end of the section we focus on studying the growth of the natural powers of absolutely $(C,\alpha)$-Ces\`aro bounded operators.

\begin{theorem}\label{residual}
Let $0<\alpha \le 1,$ $X$ be a Banach space and $T\in B(X)$ satisfying $\|T^n\|=O(n^{\alpha}).$ Then either $\displaystyle \|T^n\|=o(n^{\alpha})$ or the set
$$
\Bigl\{x\in X: \sup_N \frac{1}{k^{\alpha+1}(N)}\sum _{j=0}^N  k^{\alpha}(N-j) \|T^jx\|=\infty\Bigr\}
$$
is residual in $X$.
\end{theorem}

\begin{proof}
Suppose that $\frac{\|T^n\|}{n^{\alpha}}\not\to 0$ as $n\to\infty.$ So, there exists a positive constant $c$ such that
$$
\limsup_{n\to\infty}n^{-\alpha}\|T^n\|>c.
$$
Also, by hypothesis, there is $C>0$ such that $\|T^n\|\leq Cn^{\alpha}$ for all $n\in\N.$

For $s\in\NN$ we denote
$$
M_s=\Bigl\{x\in X: \sup_N \frac{1}{k^{\alpha+1}(N)}\sum _{j=0}^N  k^{\alpha}(N-j) \|T^jx\|>s\Bigr\}.
$$
It is clear that $M_s$ is  an open set.

First, we show that $M_s$ contains a unit vector for $s\in\NN$. Note that there exists $N\in\N$ enough large such that $\frac{cN^{\alpha}}{k^{\alpha+1}(N)}\big(\frac{1}{C\Gamma(\alpha)2^{1-\alpha}}\ln N+1\big)>s$ and $\|T^N\|>cN^{\alpha}$. Therefore we can take a unit vector $x\in X$ such that $\|T^Nx\|> cN^{\alpha}$. Using that $\|T^Nx\|\le \|T^j\|\cdot\|T^{N-j}x\|$ for $j=1,\dots,N-1,$ one gets
$$
\|T^{N-j}x\|\ge\frac{\|T^Nx\|}{\|T^j\|}\ge\frac{cN^{\alpha}}{Cj^{\alpha}}.
$$
Thus,
\begin{eqnarray*}
\frac{1}{k^{\alpha+1}(N)}\sum _{j=0}^N  k^{\alpha}(N-j) \|T^jx\|&\ge&
\frac{1}{k^{\alpha+1}(N)}\big(\sum _{j=0}^{N-1}  k^{\alpha}(N-j) \frac{cN^{\alpha}}{C(N-j)^{\alpha}}+cN^{\alpha}\big)\\
&\ge&\frac{cN^{\alpha}}{k^{\alpha+1}(N)}\big(\frac{1}{C}\sum _{j=1}^N  k^{\alpha}(j) \frac{1}{j^{\alpha}}+1\big) \\
&\ge& \frac{cN^{\alpha}}{k^{\alpha+1}(N)}\big(\frac{1}{C\Gamma(\alpha)2^{1-\alpha}}\sum _{j=1}^N  \frac{1}{j}+1 \big)\\
&\ge& \frac{cN^{\alpha}}{k^{\alpha+1}(N)}\big(\frac{1}{C\Gamma(\alpha)2^{1-\alpha}}\ln N+1\big)>s,
\end{eqnarray*}
and so $x\in M_s$.

Now, we prove that $M_s$ is dense for each $s\in\NN.$ Indeed, let $y\in X$ and $\e>0$. Let $s'>\frac{s}{\e},$ and we take $x\in M_{s'}$ with $\|x\|=1$. For each $j\in\NN$ we have
$$
\|T^j(y+\e x)\|+\|T^j(y-\e x)\|\ge
2\e \|T^jx\|,
$$
 and then
\begin{eqnarray*}
\frac{1}{k^{\alpha+1}(N)}\sum _{j=0}^N  k^{\alpha}(N-j) \|T^j(y+\e x)\|&+&
\frac{1}{k^{\alpha+1}(N)}\sum _{j=0}^N  k^{\alpha}(N-j) \|T^j(y-\e x)\| \\&\ge&
\frac{2\e}{k^{\alpha+1}(N)}\sum _{j=0}^N  k^{\alpha}(N-j) \|T^jx\|\\
&>&2\e s'>2s.
\end{eqnarray*}
Hence either $y+\e x\in M_s$ or $y-\e x\in M_s$. Since $\e>0$ was arbitrary, we conclude that $M_s$ is dense.

The Baire category theorem implies that
$$
\bigcap_{s+1}^\infty M_s=\Bigl\{x\in X: \sup_N \frac{1}{k^{\alpha+1}(N)}\sum _{j=0}^N  k^{\alpha}(N-j) \|T^jx\|=\infty\Bigr\}
$$
is a residual set.
\end{proof}

\begin{corollary}\label{ACB}
Let $0<\alpha \le 1$ and $T$ be an absolutely $(C, \alpha)$-Ces\`{a}ro bounded operator. Then $\displaystyle \|T^n\|=o(n^{\alpha})$.
\end{corollary}

\begin{proof}
Note that by the absolutely $(C, \alpha)$-Ces\`{a}ro boundedness of $T,$ we have
$$
\|T^nx\|\leq \sum _{j=0}^n  k^{\alpha}(n-j) \|T^jx\|\leq C k^{\alpha+1}(n)\|x\| \le C n^{\alpha}\|x\|,\quad C>0.
$$
By Theorem \ref{residual}, we conclude the result, since the second possibility in Theorem \ref{residual} contradicts the absolutely $(C,\alpha)$-Ces\`{a}ro boundedness of $T.$
\end{proof}

\begin{theorem}\label{kreiss}
Let $T$ be an absolutely $(C, \alpha)$-Ces\`{a}ro bounded  operator with $\alpha >1$. Then $T$ is Kreiss bounded.
\end{theorem}
\begin{proof}
For each $\beta>\alpha $ the operator  $T$ is an absolutely $(C, \beta)$-Ces\`{a}ro bounded  operator. Then, by \cite[Corollary 2.3]{S1} we get the result.

\end{proof}

The following corollary is a straightforward consequence of Theorem \ref{kreiss} and \cite[p.344]{SW} (or \cite[Proposition 2]{Shi}).

\begin{corollary}\label{On}
Let $T$ be an absolutely $(C, \alpha)$-Ces\`{a}ro bounded  operator with $\alpha >1$. Then $\|T^{n}\|=O(n)$
\end{corollary}

Also, by Theorem \ref{ejemplos} we have the next result.
\begin{corollary}
Let $T$ be an  absolutely $(C, \alpha)$-Ces\`{a}ro bounded  operator with $\alpha >1$. Then $\|T^n\|=o(n)$ or $T$
is not absolutely Ces\`{a}ro bounded.
\end{corollary}

\begin{remark}
For any $0<\beta<1,$ there exists an operator on $\ell^1(\N)$ which is absolutely $(C,\alpha)$-Ces\`aro bounded for $\alpha>\beta$ (Theorem \ref{ejemplos}) and no absolutely $(C,\gamma)$-Ces\`aro bounded for $0<\gamma\le\beta$ (Corollary \ref{mixing}, Corollary \ref{ACB}).
 It would be interesting to find an example of an absolutely $(C, \alpha)$-Ces\`{a}ro bounded operator with $\alpha >1$ such that it is not absolutely Ces\`{a}ro bounded.
\end{remark}

\bigskip


In Hilbert spaces, it is known that if there exists $\varepsilon >0$ such that $T\in\mathcal{B}(X)$ satisfies $\|T^n\|\ge Cn^{\frac{1}{2}+\varepsilon}$ for all $n$, then there exists $x\in X$ such that $\|T^nx\|\rightarrow \infty$, see by \cite[Theorem 3]{MV}. Therefore $T$ is not
absolutely $(C,\alpha)$-Ces\`{a}ro bounded for any $\alpha>0$.

As consequence of Corollary \ref{ACB} and \cite[Theorem 2.4]{BermBMP}, we have

\begin{corollary} \label{acbhilbert}
Let $H$ be a Hilbert space and $T$ be an absolutely $(C, \alpha)$-Ces\`{a}ro bounded operator with $0<\alpha \le 1$. Then $\displaystyle \lim_{n\to\infty}\displaystyle \frac{ \|T^n\|}{n^{min\{\alpha,1/2\}}}=0$
\end{corollary}

\section{Ergodic applications for the Ces\`aro means}

In this section we study the ergodic behaviour of the Ces\`aro means supposing several assumptions, both the geometry of the Banach space $X$, and spectral and growth conditions on the operator $T.$

In the first results, the reflexive property plays role. As consequence of Corollary \ref{ACB}, Corollary \ref{On} and \cite[Proposition 3.2]{Ed-dari 2}, we obtain the following result.

\begin{corollary}\label{ErgRef} Let $\alpha>0$. Every absolutely $(C, \alpha)$-Ces\`{a}ro bounded operator in a reflexive Banach space is $(C, \alpha)$-mean ergodic.
\end{corollary}
\noindent Hence, by Corollary \ref{mixing}, we have the next corollary.

\begin{corollary}
Let $\alpha>0$. There exists a mixing $(C, \alpha)$-mean ergodic operator on $\ell^p(\Bbb N)$ for $1< p <\infty$.
\end{corollary}
\noindent Results according to previous ones are shown in \cite{AS, beltran14, BermBMP} for $\alpha =1.$

\bigskip

From now on, we focus on stability results for the Ces\`aro mean differences of size $n$ and $n+1$ for bounded operators. In 1986, Y. Katznelson and L. Tzafriri proved that if $T\in \mathcal{B}(X)$ is power bounded, then $\lim_{n\to\infty}\lVert T^n-T^{n+1}\rVert=0$ if and only if $\sigma(T)\cap \partial \Bbb D \subset \{1\},$ see \cite[Theorem 1]{KT}. If $T$ is $(C,1)$-Ces\`aro bounded and $\sigma(T)\cap \partial \Bbb D = \{1\},$ but $T$ is not power bounded, then $\lVert T^n-T^{n+1}\rVert$ need not converge to zero. See examples in \cite{L, TZ}. Also, the result is not true if the condition $(C,1)$-Ces\`aro bounded is changed by Kreiss bounded, see \cite[Example 4]{N}.  Recently, in \cite{Abadias}, the author proved that if $T$ is $(C,\alpha)$-Ces\`aro bounded, with $\alpha>0,$ and $\sigma(T)\cap \partial \Bbb D \subset \{1\},$ then $\lim_{n\to\infty}\lVert M_T^{\alpha}(n+1)-M_T^{\alpha}(n)\rVert=0.$


We recall the following lemma which will be used in the proof of the main theorem of this section.

\begin{lemma}(\cite[Lemma 1]{Y}) \label{Y}
Let $\alpha >0$ and $T\in \mathcal{B}(X)$ such that $\|T^n\|=o(n^{\omega})$, where $w=min(1,\alpha).$ Then
$\displaystyle \lim_{n\to\infty} \|M_T^{\alpha }(n)(T-I)\|=0$.
\end{lemma}


\begin{theorem}\label{ergodic} If any of the following conditions is true:

\begin{enumerate}
\item [a)]  $T$ is an absolutely $(C, \alpha)$-Ces\`{a}ro bounded  operator and $0<\alpha \le 1,$

\item [b)]  $T$ is Kreiss bounded, $1\le\alpha <2$ and $m(\sigma(T)\cap \partial \Bbb D)=0$,

\item [c)] $T$ is Kreiss bounded and  $\alpha \ge 2$,

\end{enumerate}
then
$$
\displaystyle \lim_{n\to\infty}\|M_T^{\alpha }(n+1)- M_T^{\alpha }(n)\|=0.
$$
\end{theorem}
\begin{proof}
a) If $\alpha =1$, from the known identity
$$
\frac{n+2}{n+1}M_T(n+1)- M_T(n) = \frac{1}{n+1} T^{n+1},
$$
one gets
 $$
M_T(n+1)- M_T(n) = \frac{1}{n+1} T^{n+1} - \frac{1}{n+1}M_T(n+1).
$$
The result follows from \cite[Corollary 2.6]{BermBMP} and the $(C,1)$-Ces\`aro boundedness of $T.$

If $0< \alpha < 1$, note that $$
M_T^{\alpha }(n+1)- M_T^{\alpha }(n) = \frac{\alpha}{n+\alpha +1} I +\frac{n+1}{n+\alpha +1}M_T^{\alpha }(n)(T-I)+ \frac{\alpha}{n+\alpha +1}M_T^{\alpha }(n),
$$ see \cite[p.79]{Abadias}. By Corollary \ref{ACB} and Lemma \ref{Y} we have
$$
\displaystyle \lim_{n\to\infty} \|M_T^{\alpha }(n)(T-I)\|=0,
$$
 and therefore $\displaystyle \lim_{n\to\infty}\|M_T^{\alpha }(n+1)- M_T^{\alpha }(n)\|=0.$

b) Note that by Remark \ref{kreiss2} ii) we have that $\|T^n\|=o(n)$. Then by Lemma \ref{Y}, one gets  $\displaystyle \lim_{n\to\infty} \|M_T^{\alpha }(n)(T-I)\|=0$ for  $1\le\alpha <2.$ Using that the Kreiss condition implies $\lVert M_{T}(n)\rVert=O(\ln (n))$ (\cite[Theorem 6.2]{SW}), we have that $
 \lim_{n\to\infty}\|\frac{1 }{n} M_T ^{\alpha }(n)\| \le \lim_{n\to\infty}\frac{1 }{n}\sup_{0\leq j\leq n}\| M_T (j)\|=0$. By the above comments and the following identity (\cite[p.78]{Abadias})
\begin{equation*}
M_T^{\alpha }(n+1)- M_T^{\alpha }(n)= M_T^{\alpha }(n)(T-I) + \frac{\alpha}{n+1}(I-M_T^{\alpha}(n+1)) ,
\end{equation*}
we conclude the result.

c) By \cite[Theorem 6.2]{SW}, since $\alpha \ge 2$ we have
$$
 \lim_{n\to\infty}\|\frac{1 }{n} M_T ^{\alpha -1}(n)\| \le \lim_{n\to\infty}\frac{1 }{n}\sup_{0\leq j\leq n}\| M_T (n)\|=0.
$$
Doing use of the identity (\cite[p.78]{Abadias})
 $$
M_T^{\alpha }(n+1)- M_T^{\alpha }(n) = \frac{\alpha }{n+1} M_T ^{\alpha -1}(n+1) - \frac{\alpha}{n+1}M_T ^{\alpha}(n+1),
$$
we get
$$
\displaystyle \lim_{n\to\infty}\|M_T^{\alpha }(n+1)- M_T^{\alpha }(n)\|=0.
$$
\end{proof}

Since all absolutely $(C, \alpha)$-Ces\`{a}ro bounded operators are Kreiss bounded, we get the following corollary.

\begin{corollary}\label{ergodic2}
Let $T$ be an absolutely $(C, \alpha)$-Ces\`{a}ro bounded operator. If any of the following conditions is true:

\begin{enumerate}

\item [a)] $1<\alpha <2$ and $m(\sigma(T)\cap \partial \Bbb D)=0$,

\item [b)] $\alpha \ge 2$,

\end{enumerate}
then
$$
\displaystyle \lim_{n\to\infty}\|M_T^{\alpha }(n+1)- M_T^{\alpha }(n)\|=0.
$$
\end{corollary}

We deduce from Theorem \ref{ergodic} a) that  if $0<\alpha\le 1$ and  $T$ is an absolutely $(C, \alpha)$-Ces\`aro bounded operator, then  $\|M_T^{\alpha }(n)(T-I)\|$ converges to zero as $n\to \infty$. So, a natural question is if the following extended version of the Katznelson-Tzafriri theorem is true:

\begin{question} Is there $\alpha$ with  $0<\alpha\le 1$ such that if $T$ is an absolutely $(C, \alpha)$-Ces\`aro bounded operator with $\sigma(T) \cap \partial \Bbb D =\{1\},$  $\|T^n(T-I)\|$ converges to zero as $n\to \infty$?
\end{question}

On the other hand, assuming a less restrictive condition ($1<\alpha<2$), we obtain the next ergodic result.
\begin{theorem}
Let $T$ be an absolutely $(C, \alpha)$-Ces\`{a}ro bounded  operator with $1<\alpha <2$. Then $
\displaystyle \lim_{n\to\infty} \|M_T^{\alpha }(n)(T-I)^2\|=0.
$
\end{theorem}
\begin{proof}
If $1< \alpha < 2$, we have by Corollary \ref{On} that $\|T^n\|=o(n^{\alpha})$, then following the ideas in \cite[p.9 and 10]{Ed-dari 1} one gets
$$
\displaystyle \lim_{n\to\infty} \|M_T^{\alpha }(n)(T-I)^2\|=0.
$$
\end{proof}

To finish, we comment the relevance of the conditions for Theorem \ref{ergodic} in the Kreiss bounded case.

\begin{remark} The condition $m(\sigma(T)\cap \partial \Bbb D)=0$ in Theorem \ref{ergodic} b) is not redundant. Indeed, in the proof of \cite[Theorem 6]{N}, the author constructs a bounded linear operator $A$ on a Banach space of analytic functions in the unit disc such that  $A$ is Kreiss bounded and $\sigma(A)\cap \partial \Bbb D$ has positive measure. Moreover, it satisfies that \begin{equation}\label{6}\|A^n\|\geq 1+\frac{n}{2}m(\sigma(A)\cap \partial \Bbb D).\end{equation} Then the identity $$M_A(n+1)- M_A(n) = \frac{1}{n+1} A^{n+1} - \frac{1}{n+1}M_A(n+1),$$ identity \eqref{6} and \cite[Theorem 6.2]{SW} imply $$\lim_{n\to \infty}\|M_A(n+1)- M_A(n)\|\neq 0.$$
\end{remark}


\begin{thebibliography}{99}

\bibitem{Abadias} L. Abadias.
{\it A Katznelson-Tzafriri type theorem for Cesàro bounded operators}.
Studia Math., {\bf 234} (1) (2016), 59-82.

\bibitem{ALMV} L. Abadias, C. Lizama, P. J. Miana and M. P. Velasco.
 {\it Cesàro sums and algebra homomorphisms of bounded operators}.
 Israel J. Math., {\bf 216} (2016), no. 1, 471–505.

\bibitem{AS} A. Aleman and L. Suciu.
{\it On ergodic operator means in Banach spaces}.
Integr. Equ. Oper. Theory, \textbf{85} (2016), 259-287.

\bibitem{beltran14} M.J. Beltr\'an-Meneu, Operators on weighted spaces of
holomorphic functions. PhD Thesis, Universitat Polit\`ecnica de Val\`encia, 2014.

\bibitem{BermBMP} T. Berm\'{u}dez, A. Bonilla, V. Muller and A. Peris.
{\it Cesaro bounded  operators in Banach spaces}. ArXiv:1706.03638

\bibitem{Lizama1} S. Calzadillas, C. Lizama and J. G. Mesquita. {\it A unified approach to discrete fractional calculus and applications.} Preprint, 2014.

\bibitem{De00} Y. Derriennic. {\it On the mean ergodic theorem for Ces\`aro bounded operators}. Colloq. Math., {84/85} (2000), 443--455.

\bibitem{E} R. \'{E}milion.
{\it Mean-bounded operators and mean ergodic theorems}.
J. Funct. Anal., {\bf 61} (1985), no. 1, 1-14.

\bibitem{Ed-dari 1} E. Ed-dari.
{\it On the $(C,\alpha)$ uniform ergodic operators}.
Studia Math., {\bf 156} (1) (2003), 3-13.

\bibitem{Ed-dari 2} E. Ed-dari.
{\it On the $(C,\alpha)$ Cesàro bounded operators}.
Studia Math., {\bf 161} (2) (2004), 163-175.


\bibitem{GEP11} K.-G. Grosse-Erdmann and A. Peris. {\it Linear Chaos}. Springer, London, 2011.

\bibitem{HL} B. Hou and L. Luo.
{\it Some remarks on distributional chaos for bounded linear operators}. Turk. J. Math., {\bf 39} (2015), 251-258.

\bibitem{KT} Y. Katznelson and L. Tzafriri.
{\it On power bounded operators}.
J. Funct. Anal., {\bf 68} (1986), 313-328.

\bibitem{L} Z. Léka.
{\it A note on the powers of Cesàro bounded operators}.
 Czechoslovak Math. J., {\bf 60} (135) (2010, 1091-1100.

\bibitem{LSAS} Y.-C. Li, R. Sato and S.-Y. Shaw. {\it Boundednes and growth orders of means of discrete and continuous semigroups of operators}. Studia Math., {\bf 187} (1) (2008), 1-35.

\bibitem{Lizama} C. Lizama. {\it The Poisson distribution, abstract fractional difference equations, and stability.} Proc. Amer. Math. Soc., {\bf 145} (2017), no. 9, 3809–3827.


\bibitem{MSZ} A. Montes-Rodr{\'{i}}guez, J. S\'anchez-\'Alvarez and J. Zem\'anek.
{\it Uniform Abel-Kreiss boundedness and the extremal behavior of the Volterra operator}.
Proc. London Math. Soc., {\bf 91} (2005), 761-788.

\bibitem{MV} V.~M\"{u}ller and J.~Vrsovsky.
 {\it Orbits of linear operators tending to infinity}.
Rocky Mountain J. Math., \textbf{39} (2009), 219-230.

\bibitem{N}  O. Nevanlinna.
{\it Resolvent conditions and powers of operators}.
Studia Math., {\bf 145} (2) (2001), 113-134.

\bibitem{Shi} A. L. Shields.
{\it On Mobius bounded operators}.
Acta Sci. Math (Szeged), {\bf 40} (1978), 371-374.


\bibitem{SW}J. C. Strikwerda and B. A. Wade.
{\it A survey of the Kreiss matrix theorem for power bounded families of matrices and its extensions}.
Linear operators (Warsaw, 1994), 339-360, Banach Center Publ., 38, Polish Acad. Sci., Warsaw, 1997.

\bibitem{SZ} L. Suciu and J. Zem\'anek.
{\it Growth conditions on Ces\`{a}ro means of higher order}. Acta Sci. Math (Szeged), {\bf 79} (2013), 545-581.


\bibitem{S1} L. Suciu.
{\it Estimations of operator resolvent by higher order Ces\`aro means}. Results Math., {\bf 69} (2016), 457-475.



\bibitem{TZ} Y. Tomilov and J. Zem\'{a}nek. {\it A new way of constructing examples in operator ergodic theory}. Math. Proc. Cambridge Philos. Soc., {\bf 137} (2004), no. 1, 209-225.

\bibitem{Y} T. Yoshimoto. {\it Uniform and strong ergodic theorems in Banach spaces}. Illinois J. Math., {\bf 42} (1998), no. 4, 525–543; Correction, ibid {\bf 43} (1999), 800-801.

\bibitem{Zygmund} A. Zygmund. \emph{Trigonometric Series}. 2nd ed. Vols. I, II, Cambridge University Press, New York, 1959.

\end{thebibliography}
\end{document}